\documentclass[12pt]{article}
\usepackage{amssymb}
\usepackage{amsmath}
\usepackage[english]{babel}
\usepackage{amsthm}
\usepackage{graphicx}
\usepackage{caption}
\usepackage{subcaption}
\usepackage[utf8]{inputenc}
\usepackage{breqn}
\usepackage{amsfonts}
\usepackage[capposition=bottom]{floatrow}
\usepackage{epsfig}
\usepackage{xcolor}
\usepackage{cite}
\usepackage{breqn}
\usepackage{authblk}
\usepackage[utf8]{inputenc}
\usepackage{enumerate}
\usepackage{enumitem} 
\usepackage{float}
\usepackage{hyperref}

\def\bc{\begin{center}}
\def\ec{\end{center}}

\def\s2c{\vskip 2cm}
\def\bt{\begin{Theorem}}
\def\et{\end{Theorem}}
\def\bd{\begin{Definition}}
\def\ed{\end{Definition}}
\def\bl{\begin{Lemma}}
\def\el{\end{Lemma}}
\def\be{\begin{Example}}
\def\ee{\end{Example}}
\def\bcor{\begin{Corollary}}
\def\ecor{\end{Corollary}}
\def\br{\begin{Remark}}
\def\er{\end{Remark}}

\def\mysection{\setcounter{equation}{0}\section}
\newtheorem{theorem}{Theorem}
\newtheorem{Lemma}{Lemma}[section]
\newtheorem{Theorem}[Lemma]{Theorem}
\newtheorem{Definition}[Lemma]{Definition}

\newtheorem{Corollary}[Lemma]{Corollary}
\newtheorem{Remark}[Lemma]{Remark}
\newtheorem{example}[Lemma]{Example}
\allowdisplaybreaks
\usepackage{titlesec}
\usepackage{ulem} 
\titleformat{\section}
{\normalfont\Large\bfseries}{\thesection.}{0.5em}{}
\date{}
\title{\textbf{Bilinear Calder\'{o}n-Zygmund operators on Vilenkin groups}}
\author{
Adil Shafi Wani$^{1}$,
Qaiser Jahan$^{1}$\thanks{Corresponding author: \texttt{qaiser@iitmandi.ac.in}}, 
 Salman Ashraf$^{2}$
}
\affil{$^{1}$School of Mathematical and Statistical Sciences,
Indian Institute of Technology Mandi, Kamand (H.P.) 175005, India}

\affil{$^{2}$Department of Mathematics, Applied Science Cluster,
University of Petroleum and Energy Studies (UPES),
Dehradun, Uttarakhand 248007, India}

\begin{document}
\maketitle

\begin{abstract}
In this article, we study bilinear Calder\'on--Zygmund operators on a Vilenkin group $G$. As a preliminary step, we establish a Grafakos--Torres-type endpoint weak-type result in our setting. Furthermore, we prove that such operators extend to bounded bilinear mappings from $L^{p_1}(G)\times L^{p_2}(G)$ into $L^p(G)$ under the natural condition
$\frac{1}{p}=\frac{1}{p_1}+\frac{1}{p_2}.$
We then obtain a corresponding boundedness result in Morrey spaces, showing that these operators extend to bounded bilinear mappings from $\mathcal{M}_{p_1,u_1}(G)\times \mathcal{M}_{p_2,u_2}(G)$ into $\mathcal{M}_{p,u}(G)$ under suitable assumptions. These results generalize the classical bilinear estimates to the setting of Vilenkin groups.
\end{abstract}


\mysection{Introduction}
The theory of bilinear singular integral operators has become a central topic in modern harmonic analysis since the pioneering work of Coifman and Meyer~ \cite{21} and the subsequent development of multilinear Calderón–Zygmund operators by several authors, including Christ and Journé\cite{22}, Kenig and Stein \cite{23}, and later Grafakos and Torres \cite{24}. Bilinear singular integrals play an important role in several areas of analysis, including product estimates, nonlinear partial differential equations, and time-frequency analysis. In the Euclidean setting, a systematic theory of bilinear Calder\'on--Zygmund operators includes strong-type boundedness on products of Lebesgue spaces together with the endpoint weak-type implication. Further developments on multilinear Calderón–Zygmund operators and their boundedness properties on various function spaces can be found in  \cite{24,25,26,27,29,30} and the references therein.

The aim of the present article is to develop an analogue of this theory on a Vilenkin group $G$. Vilenkin groups constitute a class of zero-dimensional locally compact Abelian groups that generalize $p-$adic groups and the additive groups of local fields. These groups are totally disconnected and non-Archimedean, and provide a geometry that is different from the Euclidean case, characterized by a topology generated by a nested sequence of compact open subgroups, and they also serve as useful models for ultrametric structures arising in non-Archimedean analysis and related areas of mathematical physics. Extensive research has been conducted on various function spaces within this framework, such as weighted Hardy spaces, weighted Lebesgue spaces, Herz spaces, Herz–Hardy spaces, and weighted Triebel–Lizorkin spaces\cite{11,13,14,15,16,17,18,19}.\\

Our principal objective is to establish boundedness properties of bilinear Calder\'on--Zygmund operators on Vilenkin groups, both on products of Lebesgue spaces and in the Morrey-space setting. A key technical ingredient in our approach is an endpoint weak-type implication of Grafakos--Torres type. More precisely, we first show that if a bilinear Calder\'on--Zygmund operator $T$ satisfies
\[
T:L^p(G)\times L^q(G)\to L^{r,\infty}(G),
\qquad
\frac1p+\frac1q=\frac1r,
\]
for some $1\le p,q\le \infty$ and $0<r<\infty$, then
\[
T:L^1(G)\times L^1(G)\to L^{1/2,\infty}(G).
\]
This is the Vilenkin-group analogue of the corresponding endpoint implication established by Grafakos and Torres in the Euclidean setting, and it serves as a crucial ingredient in the proof of our main lemma. To illustrate the behavior of the endpoint that underlies our arguments, we also provide an example.

Building on this endpoint estimate, we prove that bilinear Calder\'on--Zygmund operators extend to bounded bilinear mappings
\[
L^{p_1}(G)\times L^{p_2}(G)\to L^p(G),
\qquad
1<p_1,p_2<\infty,\quad \frac1p=\frac1{p_1}+\frac1{p_2}.
\]
This yields the bilinear Calder\'on--Zygmund boundedness theorem on Vilenkin groups and extends the classical strong-type theory to a natural ultrametric setting.

We then establish a corresponding boundedness result in Morrey spaces. The Morrey spaces were introduced by Charles Morrey in\cite{3} in connection with regularity problems arising in the calculus of variations and quasi-linear elliptic partial differential equations. Compared with classical Lebesgue spaces, Morrey spaces capture local integrability properties more precisely and thus provide a refined description of local regularity. Owing to these features, they have found wide applications in partial differential equations and harmonic analysis\cite{3}. Recent studies have extended the analysis of Riesz potentials and singular integral operators to Morrey spaces over $p-$adic fields\cite{5,6}.\\

In the bilinear setting Calder\'on--Zygmund operators have been studied on generalized weighted Morrey spaces over RD-spaces\cite{34}, generalized fractional mixed Morrey spaces\cite{35} and other related spaces \cite{26,36,37}. In the present setting, we prove the boundedness of bilinear Calder\'on--Zygmund operators on the product of Morrey spaces on Vilenkin groups.\\

Under suitable structural assumptions on the underlying measure and appropriate growth and summability conditions, we show that bilinear Calder\'on--Zygmund operators extend to bounded bilinear mappings
\[
\mathcal{M}_{p_1,u_1}(G)\times \mathcal{M}_{p_2,u_2}(G)\to \mathcal{M}_{p,u}(G).
\]

The paper is organized as follows. In Section~2, we recall the necessary preliminaries on Vilenkin groups, Morrey spaces, and bilinear Calder\'on--Zygmund kernels. Section~3 is devoted to the endpoint weak-type implication and example. Section~4 contains the proof of the boundedness theorem on products of Lebesgue spaces. In the last section, we establish the corresponding Morrey space boundedness result.

\mysection{Preliminaries}
A Vilenkin group $G$ is a locally compact abelian group characterized by the existence of a strictly descending sequence of compact open subgroups $\{G_n\}_{n \in \mathbb{Z}}$, satisfying
\begin{equation}
\bigcup_{n=-\infty}^{\infty} G_n = G,
\qquad
\bigcap_{n=-\infty}^{\infty} G_n = \{0\}.
\label{eq:1}
\end{equation}
\setcounter{equation}{1}
\begin{equation}
\sup\{\operatorname{order}(G_n / G_{n+1}) : n \in \mathbb{Z}\} = B < \infty .
\label{eq:2}
\end{equation}
\setcounter{equation}{2}
Before proceeding further, we present a concrete example of a Vilenkin group.\\
Let
\[
D := \Big\{ x = (x_i)_{i\in\mathbb{Z}} : x_i \in \{0,1\} \text{ for all } i\in\mathbb{Z},
\text{ and } x_i = 0 \text{ for all } i < N_x \Big\},
\]
where \(N_x \in \mathbb{Z}\) depends on \(x\). Addition on \(D\) is defined coordinate wise modulo \(2\). A metric \(d\) on \(D\) is introduced by $d(x,y) = 2^{-n},~ \text{if } (x-y)_n = 1 \text{ and } (x-y)_i = 0 \text{ for all } i < n.$ With this metric, \(D\) becomes a locally compact abelian topological group.

For each \(n \in \mathbb{Z}\), define the subgroup
\[
D_n := \{ x \in D : x_i = 0 \text{ for all } i < n \}.
\]
The family \((D_n)_{n\in\mathbb{Z}}\) forms a decreasing sequence of compact open subgroups of \(D\)
that generates the topology of the group.

The dual group \(\Gamma_D\) can be described explicitly.
For any \(y = (y_i)_{i\in\mathbb{Z}} \in D\), define the character
\[
\gamma_y(x) := \exp\!\left( \pi i \sum_{k=-\infty}^{\infty} x_k y_{-k} \right),
\quad x \in D.
\]
Then every continuous character of \(D\) is of this form, and hence
\[
\Gamma_D = \{ \gamma_y : y \in D \}.
\]

This group serves as a fundamental example of a Vilenkin group.
Other important examples include the additive groups of \(p\)-adic fields $\mathbb{Q}_p$, where $p$ is a prime number and, more generally, non-Archimedean local fields.\\

A local field can be described as a valued field where the valuation is governed by the ultrametric inequality (also known as the non-Archimedean property):$$|x + y| \leq \max(|x|, |y|).$$
Detailed discussions regarding the properties and theory of local fields are available in the works of \cite{1,38}. The reader is referred to \cite{2} for more examples of locally compact Vilenkin groups.

Let $dx$ represent the Haar measure on $G$, normalized so that $\int_{G_0} dx = 1.$ The Haar measure of any measurable set $A \subset G$ is denoted  by $|A|$. Let $|G_n| = m_n^{-1}, n \in \mathbb{Z}.$
Then we have
\begin{equation}
2 m_n \le m_{n+1} \le B m_n, \qquad \forall n \in \mathbb{Z}.\label{eq:3}
\end{equation}
\setcounter{equation}{3}
Therefore, for any $k \in \mathbb{Z}$ and $\alpha > 0$, we obtain
\begin{equation}
\sum_{n=k}^{\infty} \frac{1}{m_n^{\alpha}} \le C \frac{1}{m_k^{\alpha}},
\qquad
\sum_{n=-\infty}^{k} m_n^{\alpha} \le C m_k^{\alpha},
\label{eq:4}
\end{equation}
\setcounter{equation}{4}
for some constant $C>0$.

We denote the collection of all measurable functions defined on $G$ as $\mathcal{M}(G)$. Furthermore, we let $\mathcal{I}$ represent the collection of cosets in $G$, defined as $\mathcal{I} = \{x + G_k : x \in G, k \in \mathbb{Z}\}$.

For any $x$, $y$ in $G$, we define the function $d: G \times G \to [0, \infty)$ by $$d(x, y) = \begin{cases} m_n^{-1}, & \text{if } x - y \in G_n \setminus G_{n+1} \\ 0, & \text{if } x = y .\end{cases}$$  This function $d$ is established as a metric that induces a topology identical to the original topology of $G$. We also write $|x| = d(x,0).$

A function $\varphi : G \to \mathbb{C}$ belongs to $\mathcal{S}(G)$ if there exist $k,l \in \mathbb{Z}$
such that $\operatorname{supp}\varphi \subset G_k$ and $\varphi$ is constant on the cosets of the
subgroup $G_l$ of $G$.

We say that a sequence $\{\varphi_n\}$ converges to $\varphi$ in $\mathcal{S}(G)$ if there exist
$k,l \in \mathbb{Z}$ such that
\[
\operatorname{supp}\varphi_n \subset G_k, \quad n \in \mathbb{Z}, \qquad
\operatorname{supp}\varphi \subset G_k,
\]
and $\varphi_n$ and $\varphi$ are constant on the cosets of $G_l$ in $G$, and
\[
\lim_{n \to \infty} \varphi_n = \varphi \quad \text{uniformly on } G.
\]

Let $\mathcal{S}'(G)$ denote the space of all continuous linear functionals on $\mathcal{S}(G)$.
We say that a sequence $\{f_n\} \subset \mathcal{S}'(G)$ converges to $f \in \mathcal{S}'(G)$ if,
for any $\varphi \in \mathcal{S}(G)$,
\[
\lim_{n \to \infty} f_n(\varphi) = f(\varphi).
\]
\begin{Definition} \label{Definition 2.1}
    Let \(G\) be a Vilenkin group equipped with  Haar measure \(dx\).
For \(1 \leq p < \infty\), the space \(L^{p}(G)\) consists of all measurable
functions \(f : G \to \mathbb{R}\) such that
\[
\|f\|_{L^{p}(G)}
= \left( \int_{G} |f(x)|^{p} \, dx \right)^{1/p} < \infty.
\]
For \(p = \infty\), the space \(L^{\infty}(G)\) is defined as the collection
of all essentially bounded measurable functions on \(G\), endowed with the norm
\[
\|f\|_{L^{\infty}(G)} = \operatorname*{ess\,sup}_{x \in G} |f(x)|.
\]
\end{Definition}
   We now recall the definition of Morrey spaces on Vilenkin groups from \cite{Ho}.
\begin{Definition} \label{Definition 2.2}
Let \(G\) be a Vilenkin group, let \(0 < p < \infty\), and let
\(u : \mathcal{I} \to (0,\infty)\) be a positive function defined on the family
\(\mathcal{I}\) of cosets of \(G\).
The Morrey space $\mathcal{M}_{p,u}(G)$ is defined as the collection of all measurable
functions \(f \in \mathcal{M}(G)\) such that
\begin{equation}
    \|f\|_{\mathcal{M}_{p,u}(G)}
:= \sup_{I \in \mathcal{I}} \frac{1}{u(I)} \,
\big\| \chi_{I} f \big\|_{L^{p}(G)} < \infty,
\label{eq:def}
\end{equation}
where \(\chi_{I}\) denotes the characteristic function of the set \(I\).
\end{Definition}
When \(0 < p \leq q < \infty\) and the function \(u\) is chosen as $u(I) = |I|^{\frac{1}{p}-\frac{1}{q}},$ the space $\mathcal{M}_{p,u}(G)$, denoted by $\mathcal{M}_{p,q}(G)$ coincides with the natural counterpart of the classical Morrey space on \(\mathbb{R}^{n}\) originally
introduced by Morrey in \cite{3}. More generally, the spaces $\mathcal{M}_{p,u}(G)$ provide analogues of the generalized Morrey spaces on \(\mathbb{R}^{n}\) studied by Nakai in \cite{4}.
In particular, these spaces include the generalized Morrey spaces on the \(p\)-adic
field \(\mathbb{Q}_p\)\cite{5,6}.
Since every non-Archimedean local field \(K\) is a Vilenkin group, the above
definition also yields the Morrey space on non-Archimedean local field $\mathcal{M}_{p,u}(K)$.
\begin{Definition} \label{Definition 2.3}
    For any locally integrable function $f$, the Hardy--Littlewood maximal operator on $G$ is defined by
\[
Mf(x)
:= \sup_{I\ni x}\frac{1}{|I|} \int_I |f(y)| \, dy,
\]
where the supremum is taken over all $I \in \mathcal{I}$ containing $x$.\\
The sharp maximal operator is defined by
\[
M^{\#}f(x)
:= \sup_{I \ni x} \frac{1}{|I|}
\int_I \bigl| f(y) - f_I \bigr| \, dy,
\]
where the supremum is taken over all $I \in \mathcal{I}$ containing $x$, and
\[
f_I := \frac{1}{|I|} \int_I f(y)\, dy
\]
denotes the average of $f$ over $I$.

For $\delta>0$, we define the $\delta$-sharp maximal operator by
\[
M^{\#}_\delta f := \bigl(M^{\#}(|f|^\delta)\bigr)^{1/\delta}.
\]
Similarly, for $\delta>0$, we define
\[
M_\delta f
:= \bigl(M(|f|^\delta)\bigr)^{1/\delta}.
\]
\end{Definition}
According to the first lemma in Section~3 of \cite{7}, the operator $M$ is bounded on $L^p(G)$ for all $p \in (1,\infty)$ and is of weak type $(1,1)$. \\

The notion of multilinear Calderón–Zygmund operators was introduced and systematically developed by Grafakos and Torres in \cite{24}. We recall that definition and adapt it to the framework of Vilenkin groups.
\begin{Definition} \label{Definition 2.4}
Let $K(x,y_1,y_2)$ be a locally integrable function defined away from the diagonal $x=y_1=y_2$ in $G\times G\times G$, which satisfies the size estimate
\begin{equation}\label{eq:size}
|K(x,y_1,y_2)| \le \frac{C}{\big(d(x,y_1)+d(x,y_2)\big)^{2}}
\end{equation}
for some $C>0$ and all $(x,y_1,y_2)\in G\times G\times G$ with $x\neq y_j$ for some $j$. Furthermore, assume that for some $\varepsilon>0$ we have the smoothness estimates
\begin{equation}\label{eq:smooth1}
|K(x,y_1,y_2)-K(x',y_1,y_2)| \le \frac{Cd(x,x')^{\varepsilon}}{\big(d(x,y_1)+d(x,y_2)\big)^{2+\varepsilon}}
\end{equation}
whenever $d(x,x') \le \tfrac{1}{2}\max\{d(x,y_1),d(x,y_2)\}$, and also that
\begin{equation}\label{eq:smooth2}
|K(x,y_1,y_2)-K(x,y_1',y_2)| \le \frac{Cd(y_1,y_1')^{\varepsilon}}{\big(d(x,y_1)+d(x,y_2)\big)^{2+\varepsilon}}
\end{equation}
whenever $d(y_1,y_1') \le \tfrac{1}{2}\max\{d(x,y_1),d(x,y_2)\}$, as well as a similar estimate with the roles of $y_1$ and $y_2$ reversed. Kernels satisfying these conditions are called bilinear Calder\'on--Zygmund kernels. A bilinear operator $T$ is said to be associated with the kernel $K$ if
\begin{equation}\label{eq:rep}
T(f_1,f_2)(x)=\int_G \int_G K(x,y_1,y_2)\,f_1(y_1)\,f_2(y_2)\,d\mu(y_1)\,d\mu(y_2)
\end{equation}
for all $f_1,f_2\in \mathcal{S}(G)$ and all
$x\notin \operatorname{supp}(f_1)\cap \operatorname{supp}(f_2)$.
\end{Definition}

In the subsequent sections, $C$ denotes a positive constant, which may vary from line to line but is always independent of the essential parameters. $\chi_I$ denotes the characteristic function of the measurable coset $I$ and for any exponent $p$ we denote its conjugate index by $p'$; that is, $\frac{1}{p}+\frac{1}{p'}=1$.
    For any \(I=x+G_l\), \(x \in G~and~l \in \mathbb{Z}\), write \(I_j=x+G_{l+j}\), $j \in \mathbb{Z}$.

\section{An Endpoint Weak-Type Estimate}
Grafakos and Torres proved the endpoint weak-type estimates of the form $L^1\times L^1$ into $L^{1/2,\infty}$ for bilinear Calderón–Zygmund operators on $R^n$ in \cite{24}. We use a slightly different approach than \cite{24}, we avoid the boundedness property of Marcinkiewicz operator similar to \cite {MAL}.\\
\begin{theorem}
     Let T be a bilinear Calderón--Zygmund operator as defined in \eqref{eq:rep}. Assume that, for some $1 \le p,q \le \infty$ and $0 < r < \infty$ satisfying
\[
\frac{1}{p} + \frac{1}{q} = \frac{1}{r},
\]
the operator $T$ extends to a bounded bilinear mapping from $L^p(G)\times L^q(G)$ into $L^{r,\infty}(G)$. Then $T$ extends to a bounded bilinear mapping from $L^1(G) \times L^1(G)$ into $L^{\frac{1}{2},\infty}(G)$.
\end{theorem}
\begin{proof}
    Fix $\lambda > 0$ and let $f_1, f_2 \in L^1(G)$. Assuming, without loss of generality, that $\|f_j\|_1 = 1$, $j = 1,2$, we need to prove that
    \begin{equation}\label{eq:estimate}
        |E_\lambda|=
|\left\{ x \in G : |T(f_1,f_2)(x)| > \lambda \right\}|
\le C \lambda^{-1/2}
    \end{equation}
for some constant $C$ independent of $f_1, f_2$, and $\lambda$.
Consider the Calderón--Zygmund decomposition of each function $f_j$ at height $\lambda^{1/2}$. Then, for $j = 1,2$, we have
\[
f_j = g_j + b_j
~\text{and}~ b_j = \sum_k b_{j,k},\]
where each $b_{j,k}$ is supported in coset $I_{j,k}$, for $k\neq k'$, the cosets $I_{j,k} ~\text{and}~I_{j,k'}$ are disjoint. Also,
\begin{enumerate}[label=(\roman*)]
    \item  $\int_{I_{j,k}} b_{j,k}(x)\,dx = 0$,
\item $\int_{I_{j,k}} |b_{j,k}(x)|\,dx \le C (\lambda)^{1/2} |I_{j,k}|$,
\item $\sum_k |I_{j,k}| \le C (\lambda)^{-1/2}$,
\item $\|g_j\|_{L^{s}} \le C (\lambda)^{1/2s'}$, for any $1\leq s\leq \infty$, where $s'$ is the conjugate exponent of $s$,
\item $\|b_j\|_{L^1} \le C$.
\end{enumerate}
Now define the level sets:
\begin{align*}
E_1 &= \{ x \in G : |T(g_1, g_2)(x)| > \lambda/4 \}, \\
E_2 &= \{ x \in G : |T(b_1, g_2)(x)| > \lambda/4 \}, \\
E_3 &= \{ x \in G : |T(g_1, b_2)(x)| > \lambda/4 \}, \\
E_4 &= \{ x \in G : |T(b_1, b_2)(x)| > \lambda/4 \}.
\end{align*}

Since
$
\big|\{ x \in G : |T(f_1,f_2)(x)| > \lambda\}\big|
\le \sum_{s=1}^4 |E_s|,
$
it suffices to prove estimate \eqref{eq:estimate} for each $E_s$. Let's first consider $E_1 = \{ x \in G : |T(g_1,g_2)(x)| > \lambda/4 \}.$ By Chebychev's inequality, and  $L^p(G) \times L^q(G) \rightarrow{L^{r,\infty}(G)}$ boundedness of $T$ with norm $D$
\begin{align*}
|E_1|
&\le \left(\frac{4}{\lambda}\right)^r\|T(g_1,g_2)\|_{L^{r,\infty}(G)}^r\\
&\le \left(\frac{4D}{\lambda}\right)^r 
\|g_1\|_{L^{p}(G)}^r 
\|g_2\|_{L^{q}(G)}^r\\
&\leq C\left(\frac{D}{\lambda}\right)^r \lambda^{\frac{r}{2p'}}\lambda^{\frac{r}{2q'}}\\
&=CD^r\lambda^{\frac{-1}{2}}.
\end{align*}
Let $I_k=x+G_{l+k}$ be a coset and let $I_{k}^*=x+G_{l+k-1}$ be an expansion of it. By \eqref{eq:3}, the expansion $I_{k}^*=x+G_{l+k-1}$ satisfies $|I_{k}^*| \leq B|I_k|$. Let $\Omega^* = \bigcup_{j=1}^2\bigcup_k I_{j,k}^*$, then, by (iii) above, we have 
\begin{equation}\label{eq;omega}
    |\Omega^*|= \sum_{j=1}^2 \sum_{k}|I_{j,k}^*|
\le B \sum_{j=1}^2 \sum_{k}|I_{j,k}|
= C\lambda^{-1/2}.
\end{equation}
Now,
\[
|E_2| = |\{ x \in G : |T(b_1,g_2)(x)| > \lambda/4 \}|\leq |x\in \Omega^*|+|x\notin \Omega^*:|T(b_1,g_2)(x)|>\lambda/4|.
\]
In view of \eqref{eq;omega}, we only need to estimate $|x\notin \Omega^*:|T(b_1,g_2)(x)|>\lambda/4|.$
Fix $x \notin \Omega^*$ and fix $I_{1,k_1}=c_{1,k_1}+G_{l+k_1}$. By the mean zero condition (i) of $b_{1,k_1}$, for any fixed $y_2\in G$, 
\[
\int_{I_{1,k_1}} K(x,y_1,y_2)\, b_{1,k_1}(y_1)\, dy_1
=
\int_{I_{1,k_1}} \bigl[ K(x,y_1,y_2) - K(x,c_{1,k_1},y_2) \bigr] b_{1,k_1}(y_1)\, dy_1.
\]

Since $x \notin I_{1,k_1}^*$ and $y_1 \in I_{1,k_1}$, the ultrametric property gives
\[
d(x,y_1) = d(x,c_{1,k_1}) \quad \text{for all } y_1 \in I_{1,k_1}.
\]

Therefore, $d(y_1,c_{1,k_1}) \le |(I_{1,k_1}|
\le d(x,c_{1,k_1}) = d(x , y_1)$ and hence 
\[d(y_1,c_{1,k_1}) \le |I_{1,k_1}|
\le \frac{1}{2} ~\text{max}~\{ d(x,y_1) ,d(x , y_2)\}.\] Applying \eqref{eq:smooth2}, we can estimate:
\begin{align*}
\left|\int_{I_{1,k_1}} K(x,y_1,y_2)\, b_{1,k_1}(y_1)\, dy_1\right|
&\le \int_{I_{1,k_1}}|b_{1,k_1}(y_1)|C \frac{d(y_1,c_{1,k_1})^{\varepsilon}}{\big(d(x,y_1)+d(x,y_2)\big)^{2+\varepsilon}}dy_1\\
&\le \int_{I_{1,k_1}}|b_{1,k_1}(y_1)|C \frac{|I_{1,k_1}|^{\varepsilon}}{\big(d(x,c_{1,k_1})+d(x,y_2)\big)^{2+\varepsilon}}dy_1.
\end{align*}
Multiplying by $|g_2(y_2)|$ and integrating over $y_2$, we obtain 
\begin{align*}
\int_G |g_2(y_2)| \left|\int_{I_{1,k_1}} K(x,y_1,y_2)\, b_{1,k_1}(y_1)\, dy_1 \right| dy_2
&\le
C\, |I_{1,k_1}|^{\varepsilon}\, \|b_{1,k_1}\|_{L^1(G)}
\int_G \frac{|g_2(y_2)|}{\bigl(d(x , c_{1,k_1}) + d(x , y_2)\bigr)^{2+\varepsilon}}\, dy_2\\
&\le
C  |I_{1,k_1}|^{\varepsilon}\, \|b_{1,k_1}\|_{L^1(G)}\|g_2\|_{L^\infty(G)}
\int_G \frac{dy_2}{\bigl(d(x , c_{1,k_1}) + d(x , y_2)\bigr)^{2+\varepsilon}}.
\end{align*}
Decompose $G$ as, 
$
G = \bigcup_{j \in \mathbb{Z}} S_j,  \text{where}~
S_j = \{ y_2 \in G : d(x,y_2) = m_j^{-1} \}.
$
Then
\begin{align*}
\int_G \frac{dy_2}{\big(d(x,c_{1,k_1}) + d(x,y_2)\big)^{2+\varepsilon}}
&= \sum_j \int_{S_j} 
\frac{dy_2}{\big(d(x,c_{1,k_1}) + m_j^{-1}\big)^{2+\varepsilon}}\\
&\le \sum_j \frac{m_j^{-1}}{\big(d(x,c_{1,k_1}) + m_j^{-1}\big)^{2+\varepsilon}}.
\end{align*}
Let $k$ be such that
$d(x,c_{1,k_1})= m_k^{-1}$, and using \eqref{eq:4}, we write 
\begin{align*}
    \int_G \frac{dy_2}{\big(d(x,c_{1,k_1}) + d(x,y_2)\big)^{2+\varepsilon}} 
    &\leq\sum_{j < k}\frac{m_j^{-1}}{\big(m_k^{-1} + m_j^{-1}\big)^{2+\varepsilon}} + \sum_{j \geq k}\frac{m_j^{-1}}{\big(m_k^{-1} + m_j^{-1}\big)^{2+\varepsilon}}\\
    &\leq\sum_{j < k}\frac{m_j^{-1}}{\big( m_j^{-1}\big)^{2+\varepsilon}} + \sum_{j \geq k}\frac{m_j^{-1}}{\big(m_k^{-1} \big)^{2+\varepsilon}}\\
    &=\sum_{j < k} m_j^{1+\varepsilon} + \frac{1}{{\big(m_k^{-1} \big)^{2+\varepsilon}}}\sum_{j \geq k}m_j^{-1}\\
    &\leq Cm_{k}^{1+\varepsilon}+\frac{1}{{\big(m_k^{-1})^{2+\varepsilon}}} Cm_k^{-1}\\
    &= \frac{C}{d(x,c_{1,k_1})^{1+\varepsilon}}.
\end{align*}
This implies 
\begin{align*}
  \int_G |g_2(y_2)| \left|\int_{I_{1,k_1}} K(x,y_1,y_2)\, b_{1,k_1}(y_1)\, dy_1 \right| dy_2=|T(b_{1,k_1},g_2)(x)|
  &\leq \frac {C  |I_{1,k_1}|^{\varepsilon}\, \|b_{1,k_1}\|_{L^1(G)}\|g_2\|_{L^\infty(G)}}{d(x,c_{1,k_1})^{1+\varepsilon}}\\
  &\leq\frac {C  |I_{1,k_1}|^{1+\varepsilon}\,\lambda} {d(x,c_{1,k_1})^{1+\varepsilon}}.
\end{align*}
Now integrating with respect to $x$ over $G\setminus I_{1,k_1}^*$, we get 
\[
\int_{G\setminus I_{1,k_1}^*} |T(b_{1,k_1}, g_2)(x)| \, dx
\le C(\lambda) |I_{1,k_1}|^{1+\varepsilon}
\int_{d(x , c_{1,k_1}) > \frac{1}{m_{l+k_1-1}}}
\frac{dx}{d(x ,c_{1,k_1})^{1+\varepsilon}}.\\
\]
Write
$
\{ x : d(x , c_{1,k_1}) > \tfrac{1}{m_{l+k_1-1}} \}
= \bigcup_{j\leq l+k_1} S_j,
$
where
$
S_j = \{ x : d(x , c_{1,k_1}) = \tfrac{1}{m_j} \}
$
and using \eqref{eq:4}, we get 
\begin{align*}
 \int_{G\setminus I_{1,k_1}^*} |T(b_{1,k_1}, g_2)(x)| \, dx
&\le C(\lambda) |I_{1,k_1}|^{1+\varepsilon}\sum_{j \leq l+k_1} \int_{S_j}
\frac{dx}{(m_j^{-1})^{1+\varepsilon}}\\
&\le  C(\lambda) |I_{1,k_1}|^{1+\varepsilon}\sum_{j \leq l+k_1}
\frac{m_j^{-1}}{(m_j^{-1})^{1+\varepsilon}}\\
&= \ C(\lambda) |I_{1,k_1}|^{1+\varepsilon}\sum_{j \leq l+k_1} m_j^{\varepsilon}\\
&\le  C(\lambda) |I_{1,k_1}|^{1+\varepsilon} \cdot m_{l+k_1}^{\varepsilon}\\
&= C(\lambda) \frac{1}{(m_{l+k_1})^{1+\varepsilon}} \cdot m_{l+k_1}^{\varepsilon}
= C(\lambda) m_{l+k_1}^{-1}.
\end{align*}
Now taking summation over all $k_1$, we get 
\begin{align*}
    \int_{G\setminus \Omega^*}|T(b_{1}, g_2)(x)|dx &= \sum_{k_1}\int_{G\setminus I_{1,k_1}^*} |T(b_{1,k_1}, g_2)(x)| \, dx\\
    &\leq C(\lambda) \sum_{k_1} m_{l+k_1}^{-1}= C(\lambda) \sum_{k_1} |I_{1,k_1}|\\
    &\leq C(\lambda)(\lambda^{-1/2})=C(\lambda)^{1/2}.
\end{align*}
Finally, applying Chebychev's inequality, we get 
$$|x\notin \Omega^*:|T(b_1,g_2)(x)|>\lambda/4|\leq \frac{4}{\lambda}\int_{G\setminus \Omega^*}|T(b_{1}, g_2)(x)|dx=C\frac{4}{\lambda}(\lambda)^{1/2}=C(\lambda)^{-1/2}.
$$
The estimate for $E_3$ is similar to that for $E_2$.
Now to estimate $E_4$, we need to estimate $|\{x\notin\Omega^*:|T(b_1,b_2)(x)| > \lambda/4\}$. Since 
$
b_j = \sum_{k_j} b_{j,k_j}, \quad j=1,2.
$

Thus,
\[
T(b_1,b_2)(x) = \sum_{k_1} \sum_{k_2} T(b_{1,k_1}, b_{2,k_2})(x).
\]
Let $
I_{j,k_j} = c_{j,k_j} + G_{l+k_j}, \quad j=1,2.
$ We fix a pair $(k_1,k_2)$.
Without loss of generality, assume that $
m_{l+k_1}^{-1} \le m_{l+k_2}^{-1}
$, that is, $|I_{1,k_1}|\leq |I_{2,k_2}|$.
Using $\int b_{1,k_1} = 0$, we write
\[
\int_{I_{1,k_1}} K(x,y_1,y_2)\, b_{1,k_1}(y_1)\, dy_1
=
\int_{I_{1,k_1}} \big[ K(x,y_1,y_2) - K(x,c_{1,k_1},y_2) \big] b_{1,k_1}(y_1)\, dy_1.
\]
Since $x \notin I_{1,k_1}^*$ and $y_1 \in I_{1,k_1}$, we have
$
d(x,y_1)= d(x,c_{1,k_1}), \forall\, y_1 \in I_{1,k_1}.
$
Similarly, $x \notin I_{2,k_2}^*$ and $y_2 \in I_{2,k_2}$ impply $
d(x,y_2)=d(x,c_{2,k_2}),  \forall\, y_2 \in I_{2,k_2}.
$
Now,
\[
d(y_1,c_{1,k_1}) \le \frac{1}{m_{l+k_1}} = |I_{1,k_1}|
\le d(x,c_{1,k_1}) = d(x,y_1).
\]

So,
\[
d(y_1,c_{1,k_1}) \le \frac{1}{2} \max\{ d(x,y_1), d(x,y_2) \}.
\]
Taking absolute values, we obtain
\begin{align*}
\left| \int_{I_{1,k_1}} K(x,y_1,y_2)\, b_{1,k_1}(y_1)\, dy_1\right|&=
    \left| \int_{I_{1,k_1}} \big[ K(x,y_1,y_2) - K(x,c_{1,k_1},y_2) \big] b_{1,k_1}(y_1)\,dy_1 \right|\\
    &\le \int_{I_{1,k_1}} 
\frac{C\, d(y_1,c_{1,k_1})^\varepsilon}{\big(d(x,y_1)+d(x,y_2)\big)^{2+\varepsilon}}
|b_{1,k_1}(y_1)|\,dy_1\\
&\le C\, (m_{l+k_1}^{-1})^{\varepsilon}
\int_{I_{1,k_1}} 
\frac{|b_{1,k_1}(y_1)|}{\big(d(x,c_{1,k_1})+d(x,c_{2,k_2})\big)^{2+\varepsilon}}
\,dy_1\\
&= \frac{C\, (m_{l+k_1}^{-1})^{\varepsilon}\|b_{1,k_1}\|_{{L^1}(G)}}{\big(d(x,c_{1,k_1})+d(x,c_{2,k_2})\big)^{2+\varepsilon}}.
\end{align*}
Multiplying by $|b_{2,k_2}(y_2)|$ and integrating over $y_2$, we obtain
\begin{align*}
    \int_{I_{2,k_2}} |b_{2,k_2}(y_2)|
\left| \int_{I_{1,k_1}} K(x,y_1,y_2) b_{1,k_1}(y_1)\,dy_1 \right| dy_2&
\le 
\frac{C\, (m_{l+k_1}^{-1})^{\varepsilon}\|b_{1,k_1}\|_{{L^1}(G)}}{\big(d(x,c_{1,k_1})+d(x,c_{2,k_2})\big)^{2+\varepsilon}}
\int_{I_{2,k_2}} |b_{2,k_2}(y_2)|\,dy_2\\
&= \frac{C\, (m_{l+k_1}^{-1})^{\varepsilon}\|b_{1,k_1}\|_{{L^1}(G)}\, \|b_{2,k_2}\|_{{L^1}(G)}}
{\big(d(x,c_{1,k_1})+d(x,c_{2,k_2})\big)^{2+\varepsilon}}.
\end{align*}
Thus,
\[
|T(b_{1,k_1}, b_{2,k_2})(x)| 
\le 
\frac{C\, (m_{l+k_1}^{-1})^{\varepsilon}\|b_{1,k_1}\|_{{L^1}(G)}\, \|b_{2,k_2}\|_{{L^1}(G)}}
{\big(d(x,c_{1,k_1})+d(x,c_{2,k_2})\big)^{2+\varepsilon}}.
\]
Since $|I_{1,k_1}|\leq |I_{2,k_2}|$, we obtain  
$$
m_{l+k_1}^{-\varepsilon}
=
m_{l+k_1}^{-\varepsilon/2} \, m_{l+k_1}^{-\varepsilon/2}
\le
m_{l+k_1}^{-\varepsilon/2} \, m_{l+k_2}^{-\varepsilon/2}.
$$ Consequently, we have 
\begin{align*}
   |T(b_{1,k_1}, b_{2,k_2})(x)|
&\le
C \,
\frac{
m_{l+k_1}^{-\varepsilon/2} \, m_{l+k_2}^{-\varepsilon/2}
\|b_{1,k_1}\|_{{L^1}(G)} \, \|b_{2,k_2}\|_{{L^1}(G)}
}{
d(x,c_{1,k_1})^{1+\varepsilon/2} \,
d(x,c_{2,k_2})^{1+\varepsilon/2}
}\\
&\le
C \,
\frac{
m_{l+k_1}^{-\varepsilon/2} \, m_{l+k_2}^{-\varepsilon/2}
(\lambda)^{1/2}|I_{1,k_1}| \,(\lambda)^{1/2} |I_{2,k_2}|
}{
d(x,c_{1,k_1})^{1+\varepsilon/2} \,
d(x,c_{2,k_2})^{1+\varepsilon/2}
}\\
&=
\frac{
C(\lambda)^{1/2}\,
m_{l+k_1}^{-(1+\varepsilon/2)} \,(\lambda)^{1/2}
m_{l+k_2}^{-(1+\varepsilon/2)}
}{
d(x,c_{1,k_1})^{1+\varepsilon/2} \,
d(x,c_{2,k_2})^{1+\varepsilon/2}
}. 
\end{align*}
summing over all $k_1$ and $k_2$, we obtain
\[
|T(b_1,b_2)(x)|
\le
C(\lambda)
\left(
\sum_{k_1}
\frac{m_{l+k_1}^{-(1+\varepsilon/2)}}
{d(x,c_{1,k_1})^{1+\varepsilon/2}}
\right)
\left(
\sum_{k_2}
\frac{m_{l+k_2}^{-(1+\varepsilon/2)}}
{d(x,c_{2,k_2})^{1+\varepsilon/2}}
\right).
\]
integrating with respect to $x$ over $G\setminus \Omega^*$, we obtain
\[
\int_{G\setminus \Omega^*}|T(b_1,b_2)(x)|
\le
C(\lambda)
\left(
\sum_{k_1}\int_{G\setminus I_{1,k_1}^*}
\frac{m_{l+k_1}^{-(1+\varepsilon/2)}}
{d(x,c_{1,k_1})^{1+\varepsilon/2}}
\right)
\left(
\sum_{k_2}\int_{G\setminus I_{2,k_2}^*}
\frac{m_{l+k_2}^{-(1+\varepsilon/2)}}
{d(x,c_{2,k_2})^{1+\varepsilon/2}}
\right).
\]
Using the same argument as in $E_2$, we have 
$$\sum_{k_i}\int_{G\setminus I_{i,k_i}^*}
\frac{m_{l+k_i}^{-(1+\varepsilon/2)}}
{d(x,c_{i,k_i})^{1+\varepsilon/2}}
\leq C\lambda^{-1/2},\quad i=1,2.$$
Finally, applying Chebychev's inequality, we obtained 
$$|x\notin \Omega^*:|T(b_1,b_2)(x)|>\lambda/4|\leq \left(\frac{4}{\lambda}\right)^{\frac{1}{2}}\int_{G\setminus \Omega^*}|T(b_{1}, b_2)(x)|dx\leq C(\lambda)^{-1/2}.
$$
This completes the proof.
\end{proof}

\medskip

The above result is sharp in the sense that $L^{1/2,\infty}(G)$ cannot be replaced by $L^{1/2}(G).$ The following example illustrates this fact.
\medskip

\begin{example}
    Let $T$ be the bilinear operator associated with the kernel
\[
K(x,y_1,y_2) = \frac{1}{(d(x,y_1)+d(x,y_2))^2}, \quad x \neq y_1, y_2,
\]
on a Vilenkin group $G$. We will show that this operator maps $L^1(G)\times L^1(G)$ into $L^{1/2,\infty}(G)$, but it does not map $L^1(G)\times L^1(G)$ into $L^{1/2}(G)$.
Let $f_1 = f_2 = \chi_I$, where $I \subset G$ is a fixed coset. Then
\[
T(f_1,f_2)(x)
= \int_G \int_G K(x,y_1,y_2) f_1(y_1) f_2(y_2)\, dy_1dy_2
= \int_I \int_I \frac{1}{(d(x,y_1)+d(x,y_2))^2}\, dy_1dy_2.
\]

Now fix $x, x'\notin I$. It is easy to verify that $K$ satisfies the size condition
\[
|K(x,y_1,y_2)| \le \frac{C}{(d(x,y_1)+d(x,y_2))^2}.
\]
Since $y_1, y_2$ represent the same coset $I$ and $x\notin I$ we get $d(x,y_1)=d(x,y_2).$
Moreover, if $d(x,x') \le \frac{1}{2} \max\{d(x,y_1), d(x,y_2)\}$, then by the isosceles triangle property of the ultrametric spaces \[
d(x,y_i) = d(x',y_i), \quad i=1,2,
\]
and hence
\[
K(x,y_1,y_2) = K(x',y_1,y_2).
\]
Similarly, the regularity conditions in the $y_1$ and $y_2$ variables also hold. Therefore, $K$ is a bilinear Calderón--Zygmund kernel. Hence, for $y_1,y_2 \in I$,\[
T(f_1,f_2)(x)
= \int_I \int_I \frac{1}{(d(x,y_1)+d(x,y_2))^2}\, dy_1dy_2.
\]
\[
T(f_1,f_2)(x) = \frac{C}{d(x,y_1)^2}.
\]
Consequently,
\[
|\{x \in G : |T(f_1,f_2)(x)| > \lambda\}|
= |\{x \in G : d(x,y_1) < C \lambda^{-1/2}\}|
=C \lambda^{-1/2},
\]
which implies that $T(f_1,f_2) \in L^{1/2,\infty}(G)$.

On the other hand,
\[
|T(f_1,f_2)(x)|^{1/2} = \frac{C}{d(x,y_1)},
\]
and hence
\[
\int_G |T(f_1,f_2)(x)|^{1/2}\, dx
=\int_G \frac{C}{d(x,y_1)}\, dx = \infty.
\]
\end{example}

\section{Boundedness on the Product of Lebesgue Spaces}
In this section we prove the boundedness theorem of bilinear Calder\'on--Zygmund operators on the products of Lebesgue spaces on Vilenkin groups. We also establish the necessary lemmas to prove our main theorem. We avoid the proof of Lemma~\ref{Lemma 3.1} because it is an easy consequence of the Euclidean case, see \cite{Wan}.

\begin{Lemma} \label{Lemma 3.1}
Let \(1<p,p_1,p_2<\infty\) such that
\[
\frac{1}{p}=\frac{1}{p_1}+\frac{1}{p_2}.
\]
If $f \in L^{p_1}(G)$ and $g \in L^{p_2}(G)$, then $fg \in L^p(G)$ and
\[
\|fg\|_{L^p(G)} \le \|f\|_{L^{p_1(G)}} \|g\|_{L^{p_2}(G)},
\quad \text{for every } f\in L^{p_1}(G),\ g\in L^{p_2}(G).
\]
\end{Lemma}
\begin{Lemma} \label{Lemma 3.2}
Let $1<p<\infty$. Then there exists a constant $C_p>0$ such that for every
$f \in L^p(G)$,
\begin{equation}\label{eq:FS-strong}
\|M_\delta f\|_{L^p(G)} \leq C_p \, \|M^{\#}_\delta f\|_{L^p(G)}.
\end{equation}
\end{Lemma}
The proof of \eqref{eq:FS-strong} has been done in \cite{9} for $M~\text{and}~ M^{\#}$ over $L^P(\mathbb{R}^n)$. The above extension involving $M_\delta~\text{and}~ M^{\#}_\delta$ can be easily verified from the arguments therein for $L^p(G).$\\

\begin{Lemma} \label{Lemma 3.3}
    Let $0<\delta<\frac12$, and let $T$ be a bilinear Calderón--Zygmund operator as defined in \eqref{eq:rep} and suppose that the measure $dx$ satisfies the condition \eqref{eq:3}. Then there exists a constant $C>0$ such that for all $f_i \in L^\infty_c(G)$, $i=1,2$ and $x \in G$
\begin{equation}\label{eq:sharp}
M^{\#}_\delta\bigl(T(f_1,f_2)\bigr)(x)
\leq C\, M(f_1)(x)\, M(f_2)(x),
\end{equation}
\end{Lemma}
\begin{proof}
    Let $I=x+G_l$ be a fixed coset. To prove inequality~\eqref{eq:sharp}, it suffices to show that
\begin{equation}\label{eq:goal}
\left(
\frac{1}{|I|}
\int_I
\Big|
|T(f_1,f_2)(y)|^\delta - |C_I|^\delta
\Bigr| \, dy
\right)^{1/\delta}
\leq
C\, M(f_1)(x)\, M(f_2)(x).
\end{equation}
for some constant $C_I$ to be determined.
Using the elementary inequality
\[
\bigl||a|^\delta - |b|^\delta\bigr|
\leq |a-b|^\delta,
\qquad 0<\delta <1,
\]
we see that it is enough to prove
\begin{equation}\label{eq:goal2}
\left(
\frac{1}{|I|}
\int_I
\bigl|
T(f_1,f_2)(y) - C_I
\bigr|^\delta \, dy
\right)^{1/\delta}
\leq
C\, M(f_1)(x)\, M(f_2)(x).
\end{equation}
We decompose
\[
f_i = f_i^0 + f_i^\infty,
\qquad
f_i^0 = f_i \chi_{I_{-1}},
\qquad
f_i^\infty = f_i \chi_{I_{-1}^c},
\quad i=1,2.
\]
We shall see that the choice
\[
C_I = T(f_1^\infty,f_2^\infty)(x)
\]
is appropriate. Then, for any $y \in G$, we write
\begin{align*}
\left(
\frac{1}{|I|}
\int_{I}
\bigl|
T(f_1,f_2)(y) - C_{I}
\bigr|^\delta
\, dy
\right)^{1/\delta} 
&\leq
\left(
\frac{1}{|I|}
\int_{I}
\bigl|
T(f_1^{0},f_2^{0})(y)
\bigr|^\delta
\, dy
\right)^{1/\delta} \\
&\quad +
\left(
\frac{1}{|I|}
\int_{I}
\bigl|
T(f_1^{0},f_2^{\infty})(y)
\bigr|^\delta
\, dy
\right)^{1/\delta} \\
&\quad +
\left(
\frac{1}{|I|}
\int_{I}
\bigl|
T(f_1^{\infty},f_2^{0})(y)
\bigr|^\delta
\, dy
\right)^{1/\delta} \\
&\quad +
\left(
\frac{1}{|I|}
\int_{I}
\bigl|
T(f_1^{\infty},f_2^{\infty})(y)
- T(f_1^{\infty},f_2^{\infty})(x)
\bigr|^\delta
\, dy
\right)^{1/\delta} \\
&=: F_1 + F_2 + F_3 + F_4 .
\end{align*}
To estimate $F_{1}$, we use the following Kolmogorov inequality from \cite{10}:
Let $(X,\mu)$ be a probability measure space and let $0<p<q<\infty$.
Then there exists a constant $C=C_{p,q}>0$ such that for every measurable
function $f$,
\begin{equation}\label{eq:weak-strong}
\|f\|_{L^p(\mu)} \leq C \, \|f\|_{L^{q,\infty}(\mu)}.
\end{equation}
To apply Kolmogorov's inequality, we normalize the measure on $I$. 
Define $d\nu := \frac{1}{|I|}dy$, so that $(I,\nu)$ is a probability space. Then applying Kolmogorov estimate with $p=\delta$ and $q=1/2$, together with the boundedness of $T:L^1(G)\times L^1(G)\rightarrow{L^{1/2,\infty}(G)}$, we get
\begin{align*}
F_1
=
\left(
\frac{1}{|I|}
\int_{I}
\bigl|
T(f_1^{0},f_2^{0})(y)
\bigr|^\delta
\, dy
\right)^{1/\delta} 
&= \|T(f_1^0, f_2^0)\|_{L^\delta(I,d\nu)}\\
&\leq
C \,
\| T(f_1^{0},f_2^{0}) \|_{L^{1/2,\infty}
\left(I, \frac{dy}{|I|}\right)} \\
& =
\frac{C}{|I|^2} \,
\| T(f_1^{0},f_2^{0}) \|_{L^{1/2,\infty}
\left(I, dy\right)} \\
&\le \frac{C}{|I|^2} \|f_1^0\|_{L^1(G)} \|f_2^0\|_{L^1(G)}\\
&\le \frac{C}{|I|^2} \|f_1\|_{L^1(I_{-1})} \|f_2\|_{L^1(I_{-1})}\\
&\leq
C
\prod_{i=1}^2
\frac{1}{|I_{-1}|}
\int_{I_{-1}}
|f_i(z_i)| \, dz
_i\\
&\leq
C
\prod_{i=1}^2
M(f_i)(x).
\end{align*}
The second last inequality above is due to the fact that $|I_{k-1}|\leq B|I_{k}|.$ 
Now we estimate $F_2$. For any $y \in I$, using \eqref{eq:size}, Definition \ref{Definition 2.3} and inequality \eqref{eq:3}, we have
\begin{align*}
|T(f_1^{0},f_2^{\infty})(y)|
&\leq
\int_G \int_G
|K(y,z_1,z_2)|
\, |f_1^{0}(z_1)| \, |f_2^{\infty}(z_2)|
\, dz_1\, dz_2\\
&\leq
C \int_{I_{-1}} \int_{G\setminus I_{-1}}
\frac{|f_1(z_1)|\,|f_2(z_2)|}
{\bigl[d(y,z_1)+d(y,z_2)\bigr]^2}
\, dz_1\, dz_2 \\
&\leq
C \int_{I_{-1}} |f_1(z_1)|\, dz_1
\int_{G\setminus I_{-1}}
\frac{|f_2(z_2)|}{(d(y,z_2)^2}
\, dz_2.
\end{align*}

We decompose $G\setminus I_{-1}$ as
\[
G\setminus I_{-1} = \bigcup_{j=-\infty}^{-1} \bigl(I_{j-1}\setminus I_j\bigr).
\]
Hence,
\begin{align*}
|T(f_1^{0},f_2^{\infty})(y)|
&\leq
C \int_{I_{-1}} |f_1(z_1)|\, dz_1
\sum_{j=-\infty}^{-1}
\int_{I_{j-1}\setminus I_j}
\frac{|f_2(z_2)|}{(d(y,z_2)^2}
\, dz_2 \\
&\leq
C \int_{I_{-1}} |f_1(z_1)|\, dz_1
\sum_{j=-\infty}^{-1}
\frac{1}{(d(y,z_2))^2}
\int_{I_{j-1}\setminus I_j} |f_2(z_2)|\, dz_2.\\
&\leq
C
\frac{|I_{-1}|}{|I_{-1}|}
\int_{I_{-1}} |f_1(z_1)|\, dz_1
\sum_{j=-\infty}^{-1}
\frac{1}{(d(y,z_2))^2}
\int_{I_{j-1}\setminus I_j} |f_2(z_2)|\, dz_2 \\
\end{align*}
Since for any $z_2\in I_{j-1}\setminus I_j,~j\in\mathbb{Z}\cap(-\infty ,0)$ and $y\in I,$ we have $z_2-x\in G_{l+j-1}\setminus G_{l+j}$ and $y-x\in G_{l}.$ As 
$2 m_n \le m_{n+1} \le B m_n, \forall n \in \mathbb{Z}.$ We obtain
\begin{align*}
d(z_2,y)&\geq d(z_2,x)-d(y,x)\\
&=\frac{1}{m_{l+j-1}}-\frac{1}{m_{l}}\\
&\geq \frac{1}{m_{l+j-1}}-\frac{1}{2^{1-j}}\frac{1}{m_{l+j-1}}\\
&\geq \frac{1}{2}\frac{1}{m_{l+j-1}}=\frac{1}{2}|I_{j-1}|.
\end{align*}
Also, since $|f_2(z_2)|\geq 0,$ and $I_{j-1}\setminus I_j \subset I_{j-1}$, we have
\begin{align*}
|T(f_1^{0},f_2^{\infty})(y)|
&\leq
C
\frac{|I_{-1}|}{|I_{-1}|}
\int_{I_{-1}} |f_1(z_1)|\, dz_1
\sum_{j=-\infty}^{-1}
\frac{1}{|I_{j-1}|^2}
\int_{I_{j-1}} |f_2(z_2)|\, dz_2 \\
|T(f_1^{0},f_2^{\infty})(y)|
&\leq
C M(f_1)(x) M(f_2)(x)\sum_{j=-\infty}^{-1}\frac{|I_{-1}|}{|I_{j-1}|}\\
&\leq
C M(f_1)(x) M(f_2)(x)
\end{align*}
Therefore,
\[
F_2
=
\left(
\frac{1}{|I|}
\int_{I}
|T(f_1^{0},f_2^{\infty})(y)|^\delta
\, dy
\right)^{1/\delta}
\leq
C M(f_1)(x) M(f_2)(x).
\]

By the same argument, $F_3$ admits the same estimate:
\[
F_3 \leq C M(f_1)(x) M(f_2)(x).
\]
Next, we estimate $F_4$. For $y \in I=x+G_l$, applying \eqref{eq:3}, \eqref{eq:smooth1},
and Definition \ref{Definition 2.3}, we obtain
\begin{align*}
|T(f_1^{\infty}&,f_2^{\infty})(y) - T(f_1^{\infty},f_2^{\infty})(x)|\\&\leq \int_G \int_G
|K(y,z_1,z_2) - K(x,z_1,z_2)|
\, |f_1^{\infty}(z_1)|\, |f_2^{\infty}(z_2)|
\, dz_1\, dz_2 \\
&\leq
C
\int_{G} \int_{G}
\frac{d(x,y)^\epsilon}
{[d(y,z_1)+d(y,z_2)]^{2+\epsilon}}
\, |f_1^{\infty}(z_1)|\, |f_2^{\infty}(z_2)|
dz_1\, dz_2.\\
&\leq C
\int_{G\setminus I_{-1}} \int_{G\setminus I_{-1}}
\frac{d(x,y)^\epsilon}
{[d(y,z_1)+d(y,z_2)]^{\epsilon}}
\frac{\, |f_1(z_1)|\, |f_2(z_2)|}{[d(y,z_1)+d(y,z_2)]^2}
dz_1\, dz_2.\\
 &\le C \sum_{j_{1}=-\infty}^{-1}
\sum_{j_{2}=-\infty}^{-1}
\left\{\int_{I_{j_{1}-1}\setminus I_{j_{1}}}\int_{I_{j_{2}-1}\setminus I_{j_{2}}}
\frac{d(x,y)^\epsilon}
{[d(y,z_1)+d(y,z_2)]^{\epsilon}}\frac{|f_1(z_1)||f_2(z_2)|}{(d(y,z_1)+d(y,z_2))^2}
\,dz_1\,dz_2\right\}\\
&\le C \prod_{i=1}^{2}
\left(
\sum_{j_{i}=-\infty}^{-1}
\int_{I_{j_{i}-1}\setminus I_{j_{i}}}
\frac{d(x,y)^\epsilon}
{(d(y,z_i))^{\epsilon}}\frac{|f_i(z_i)|}{d(y,z_i)}
\,dz_i\right)\\
&\le C \prod_{i=1}^{2}
\left(
\sum_{j_{i}=-\infty}^{-1}
\frac{(\frac{1}{m_l})^\epsilon}{(\frac{1}{m_{l+j_i-1}})^\epsilon}\frac{1}{|I_{j_{i}-1}|}
\int_{I_{j_{i}-1}}
|f_i(z_i)|\,dz_i 
\right)
\end{align*}
Hence,
\begin{align*}
|T(f_1^{\infty},f_2^{\infty})(y) - T(f_1^{\infty},f_2^{\infty})(x)|
&\leq
C
\prod_{i=1}^2 M(f_i)(x)
\sum_{j_i=-\infty}^{-1} \frac{(\frac{1}{m_l})^\epsilon}{(\frac{1}{m_{l+j_i-1}})^\epsilon}
 \\
&\leq
C M(f_1)(x) M(f_2)(x).
\end{align*}
Therefore,
\begin{align*}
F_4 &\leq C M(f_1)(x) M(f_2)(x).
\end{align*}
Combining the estimates for $F_1$, $F_2$, $F_3$, and $F_4$, the proof is complete.

\end{proof}
\medskip
\begin{theorem} \label{Theorem 1}
Let \(1<p,p_1,p_2<\infty\) and let \(p\) satisfy
$\frac{1}{p}=\frac{1}{p_1}+\frac{1}{p_2}.$
Suppose that \(T\) is a bilinear Calderón--Zygmund operator as defined in \eqref{eq:rep}.
Then there exists a constant \(C>0\) such that
\[
\|T(f_1,f_2)\|_{L^{p}(G)}
\le
C\,\|f_1\|_{L^{p_1}(G)}\,\|f_2\|_{L^{p_2}(G)}
\]
for all \(f_1\in L^{p_1}(G)\) and \(f_2\in L^{p_2}(G)\).
\end{theorem}
\begin{proof}
Let $f_1, f_2 \in L_c^\infty(G)$. Using Lemmas \ref{Lemma 3.1},\ref{Lemma 3.2} together with inequality \eqref{eq:sharp} and the boundedness of $M$ on $L^p(G)$. We obtain
\begin{align*}
 \|T(f_1,f_2)\|_{L^p(G)}& \le C\|M^{\#}_\delta (T(f_1,f_2)) \|_{L^p(G)}\\
  &\le  C\, \|M(f_1)(x)\, M(f_2)(x)\|_{L^p(G)}\\
  & \le
C\,\|M(f_1)\|_{L^{p_1}(G)}\,
\|M(f_2)\|_{L^{p_2}(G)}\\
&\le
C\,\|f_1\|_{L^{p_1}(G)}\,
\|f_2\|_{L^{p_2}(G)}.
\end{align*}
Since $L^{\infty}_c(G)$ is dense in $L^{p_i}(G), i=1,2,~for \quad 1<p_i<\infty$, for any $f_1 \in L^{p_1}(G)$ and $f_2 \in L^{p_2}(G)$ there exist sequences 
$\{f_1^k\}$ and $\{f_2^k\}$ in $L_c^\infty(G)$ such that
\[
\|f_1^k - f_1\|_{L^{p_1}(G)} \to 0 
\quad \text{and} \quad
\|f_2^k - f_2\|_{L^{p_2}(G)} \to 0 .
\]

We show $T_k = \{T(f_1^k,f_2^k)\}$ is Cauchy in $L^p(G)$. Indeed,
\begin{align*}
\|T_k - T_\ell\|_{L^p(G)}
&= \|T(f_1^k,f_2^k) - T(f_1^\ell,f_2^\ell)\|_{L^p(G)}\\
&= \|T(f_1^k-f_1^\ell,f_2^k) + T(f_1^\ell,f_2^k-f_2^\ell)\|_{L^p(G)}.  
\end{align*}
Using the same argument as above, we have
\[\|T_k - T_\ell\|_{L^p(G)}
\le C\left(
\|f_1^k-f_1^\ell\|_{L^{p_1}(G)}\|f_2^k\|_{L^{p_2}(G)}
+
\|f_2^k-f_2^\ell\|_{L^{p_2}(G)}\|f_1^\ell\|_{L^{p_1}(G)}
\right).
\]

Since $\{f_1^k\}$ and $\{f_2^k\}$ are convergent, they are bounded. Hence
\[
\sup_k \|f_2^k\|_{L^{p_2}(G)} < \infty,
\qquad
\sup_\ell \|f_1^\ell\|_{L^{p_1}(G)} < \infty.
\]

As $k,\ell \to \infty$, 
$\|f_1^k - f_1^\ell\|_{L^{p_1}(G)} \to 0
\quad \text{and} \quad
\|f_2^k - f_2^\ell\|_{L^{p_2}(G)} \to 0.
$

Consequently,
$$
\|T_k - T_\ell\|_{L^p(G)} \to 0,$$
that is,
$$\|T(f_1^k,f_2^k) - T(f_1^\ell,f_2^\ell)\|_{L^p(G)} \to 0.
$$

Thus $T_k = \{T(f_1^k,f_2^k)\}$ is a Cauchy sequence in $L^p(G)$. Since $L^p(G)$ is complete for $p>1$, the sequence converges to some limit in $L^p(G)$. We denote this limit by $T(f_1,f_2).$ For each $k$, we have $$\|T(f_1^k,f_2^k)\|_{L^p(G)} \leq C\|f_1^k\|_{L^{p_1}(G)}\|f_2^k\|_{L^{p_2}(G)}.$$ 
Letting $k\to\infty$, we obtain the desired result.
\end{proof}

\section{Boundedness on the Product of Morrey Spaces}
In this section, we establish the boundedness of bilinear Calder\'on--Zygmund operators on products of Morrey spaces. The argument is based on the corresponding boundedness result on product of Lebesgue spaces.
\begin{theorem} \label{Theorem 2}
Let \(T\) be a bilinear Calder\'{o}n-Zygmund operator as defined in \eqref{eq:rep} and \(u_1,u_2 : \mathcal{I} \to (0,\infty)\). Suppose that \(1<p,p_1,p_2<\infty\) satisfy $\frac{1}{p}=\frac{1}{p_1}+\frac{1}{p_2},$
and that the measure \(dx\) satisfies the condition \eqref{eq:3}  and
\begin{align} \label{Equation 4.16}
    u_i(I_{j-1}) &\le C\, u_i(I_j), \qquad j\in\mathbb{Z}.
\end{align}

If there exists a constant \(C>0\) such that for all \(I\in\mathcal{I}\),
\begin{align} \label{Equation 4.17}
   \sum_{j=-\infty}^{-1}
\frac{|I|^{1/p_i}}{|I_{j-1}|^{1/p_i}}
\,u_i(I_{j-1})
&\le C\,u_i(I),
\qquad i=1,2, 
\end{align}
and \(u=u_1u_2\), then
\[
\|T(f_1,f_2)\|_{\mathcal{M}_{p,u}(G)} 
\le
C\,\|f_1\|_{\mathcal{M}_{p_1,u_1}(G)}\,
\|f_2\|_{\mathcal{M}_{p_2,u_2}(G)},
\]
where
\[
f_i \in \mathcal{M}_{p_i,u_i}(G), \qquad i=1,2.
\]
\end{theorem}
\begin{proof}
Let $I=x+G_l$ be a fixed coset. Decompose \(f_i\) as  \[f_i=f_i^{0}+f_i^{\infty}, \qquad i=1,2,\] where $f_i^{0}=f_i\chi_{I_{-1}}~\text{and}~ f_i^{\infty}=f_i\chi_{G\setminus I_{-1}.}$ Then, using \eqref{eq:def} and the triangle inequality, we write 
\begin{align*}
\sup_{I\in\mathcal{I}} \frac{1}{u(I)}\|\chi_I T(f_1,f_2)\|_{L^{p}(G)}
&\le \sup_{I\in\mathcal{I}} \frac{1}{u(I)}\|\chi_I T(f_1^{0},f_2^{0})\|_{L^{p}(G)}\\
&+\sup_{I\in\mathcal{I}} \frac{1}{u(I)}\|\chi_I T(f_1^{0},f_2^{\infty})\|_{L^{p}(G)}\\
&+\sup_{I\in\mathcal{I}} \frac{1}{u(I)}\|\chi_I T(f_1^{\infty},f_2^{0})\|_{L^{p}(G)}\\
&+\sup_{I\in\mathcal{I}} \frac{1}{u(I)}\|\chi_I T(f_1^{\infty},f_2^{\infty})\|_{L^{p}(G)}\\
&=:E_1+E_2+E_3+E_4
\end{align*}
From \eqref{eq:def}, Theorem \ref{Theorem 1}, and the identity $u(I)=u_1(I)u_2(I)$,
\begin{align*}
E_1&=\sup_{I\in\mathcal{I}} \frac{1}{u(I)}\|\chi_I T(f_1^{0},f_2^{0})\|_{L^{p}(G)}\\
&\le C\sup_{I\in\mathcal{I}} \frac{1}{u(I)}\|\chi_{I}f_1\|_{L^{p_1}(G)}\|\chi_{I}f_2\|_{L^{p_2}(G)}\\
&\le C\sup_{I\in\mathcal{I}}\frac{1}{u(I)}u_1(I)\|f_1\|_{\mathcal{M}_{p_1,u_1}(G)}u_2(I)\|f_2\|_{\mathcal{M}_{p_2,u_2}(G)}\\
&\le C\|f_1\|_{\mathcal{M}_{p_1,u_1}(G)}\|f_2\|_{\mathcal{M}_{p_2,u_2}(G)}
\end{align*}
To estimate \(E_2\),  we first estimate $|T(f_1^{0},f_2^{\infty})(y)|$. Let \(y\in I\),
by Definition \ref{Definition 2.4}, we obtain
\begin{align*}
|T(f_1^{0},f_2^{\infty})(y)|
&\le  C \int_G\int_G\frac{|f_1^{0}(z_1)||f_2^{\infty}(z_2)|}{(d(y,z_1)+d(y,z_2))^{2}}\,dz_1dz_2\\
&\le  C\int_{I_{-1}}\int_{G\setminus {I_{-1}}}\frac{|f_1(z_1)||f_2(z_2)|}{(d(y,z_1)+d(y,z_2))^{2}}\,dz_1dz_2\\
&\le  C\sum_{j=-\infty}^{-1}\int_{I_{-1}}\int_{I_{j-1}\setminus I_j}\frac{|f_1(z_1)||f_2(z_2)|}
{(d(y,z_1)+d(y,z_2))^{2}}\,dz_1dz_2\\
&\le  C\int_{I_{-1}}|f_1(z_1)|\,dz_1\sum_{j=-\infty}^{-1}\int_{I_{j-1}\setminus I_j}\frac{|f_2(z_2)|}{(d(y,z_2))^{2}}\,dz_2.
\end{align*}
Now applying Hölder's Inequality with $\frac{1}{p_1} + \frac{1}{p_1'} = 1$ and $\frac{1}{p_2} + \frac{1}{p_2'} = 1$ we get
\begin{align*}
|T(f_1^{0},f_2^{\infty})(y)|
&\le  C \|\chi_{I_{-1}} f_1\|_{L^{p_1}(G)} \|\chi_{I_{-1}}\|_{L^{p_1'}(G)}
\sum_{j=-\infty}^{-1}
\frac{1}{(d(y,z_2))^2}
\|\chi_{I_{j-1}} f_2\|_{L^{p_2}(G)}\|\chi_{I_{j-1}}\|_{L^{p_2'}(G)}\\
&\le  C \|\chi_{I_{-1}} f_1\|_{L^{p_1}(G)}
\frac{u_1(I_{-1})}{u_1(I_{-1})}
\|\chi_{I_{-1}}\|_{L^{p_1'}(G)}
\sum_{j=-\infty}^{-1}
\frac{1}{(d(y,z_2))^2}
\frac{u_2(I_{j-1})}{u_2(I_{j-1})}
\|\chi_{I_{j-1}} f_2\|_{L^{p_2}(G)}
\|\chi_{I_{j-1}}\|_{L^{p_2'}(G)}\\
&\le  C \|f_1\|_{\mathcal{M}_{p_1,u_1}(G)}
\|f_2\|_{\mathcal{M}_{p_2,u_2}(G)}
u_1(I_{-1})
\|\chi_{I_{-1}}\|_{L^{p_1'}(G)}
\sum_{j=-\infty}^{-1}
\frac{1}{(d(y,z_2))^2}
u_2(I_{j-1})
\|\chi_{I_{j-1}}\|_{L^{p_2'}(G)}\\
&\le  C \|f_1\|_{\mathcal{M}_{p_1,u_1}(G)}
\|f_2\|_{\mathcal{M}_{p_2,u_2}(G)} \frac{1}{|I_{-1}|}\|\chi_{I_{-1}}\|_{L^{p_1'}(G)}\|\chi_{I_{-1}}\|_{L^{p_1}(G)}u_1(I_{-1})
\frac{|I_{-1}|}{\|\chi_{I_{-1}}\|_{L^{p_1}(G)}}\\
&\times
\sum_{j=-\infty}^{-1}
\frac{1}{d(y,z_2)}
\|\chi_{I_{j-1}}\|_{L^{p_2'}(G)}\|\chi_{I_{j-1}}\|_{L^{p_2}(G)}\frac{1}{\|\chi_{I_{j-1}}\|_{L^{p_2}(G)}}\frac{u_2(I_{j-1})}{d(y,z_2)}
\end{align*}
Since $z_2\in I_{j-1}\setminus I_j,~j\in\mathbb{Z}\cap (-\infty,0)$ and $y\in I,$ we have $z_2-x\in G_{l+j-1}\setminus G_{l+j}$ and $y-x\in G_{l}.$ Using the condition  $2 m_n \le m_{n+1} \le B m_n, \forall n \in \mathbb{Z}.$ We obtain
\begin{align*}
d(z_2,y)&\geq d(z_2,x)-d(y,x)\\
&=\frac{1}{m_{l+j-1}}-\frac{1}{m_{l}}\\
&\geq \frac{1}{m_{l+j-1}}-\frac{1}{2^{1-j}}\frac{1}{m_{l+j-1}}\\
&\geq \frac{1}{2}\frac{1}{m_{l+j-1}}=\frac{1}{2}|I_{j-1}|.
\end{align*}
So we have,
\begin{align*}
|T(f_1^{0},f_2^{\infty})(y)|
&\le  C \|f_1\|_{\mathcal{M}_{p_1,u_1}(G)}
\|f_2\|_{\mathcal{M}_{p_2,u_2}(G)}
\frac{1}{|I_{-1}|}\|\chi_{I_{-1}}\|_{L^{p_1'}(G)}\|\chi_{I_{-1}}\|_{L^{p_1}(G)}u_1(I_{-1})
\frac{|I_{-1}|}{\|\chi_{I_{-1}}\|_{L^{p_1}(G)}}\\
&\times
\sum_{j=-\infty}^{-1}
\frac{1}{|I_{j-1}|}
\|\chi_{I_{j-1}}\|_{L^{p_2'}(G)}\|\chi_{I_{j-1}}\|_{L^{p_2}(G)}\frac{1}{\|\chi_{I_{j-1}}\|_{L^{p_2}(G)}}\frac{u_2(I_{j-1})}{|I_{j-1}|}\\
&\le C \|f_1\|_{\mathcal{M}_{p_1,u_1}(G)}
\|f_2\|_{\mathcal{M}_{p_2,u_2}(G)}
\frac{u_1(I_{-1})|I_{-1}|}{\|\chi_{I_{-1}}\|_{L^{p_1}(G)}}
\left\{
\sum_{j=-\infty}^{-1}
\frac{1}{|I_{j-1}|}
\frac{u_2(I_{j-1})}{\|\chi_{I_{j-1}}\|_{L^{p_2}(G)}}
\right\}.
\end{align*}
The last inequality above is due to the fact that $\frac{1}{|I_{-1}|}\|\chi_{I_{-1}}\|_{L^{p_i'}(G)}\|\chi_{I_{-1}}\|_{L^{p_i}(G)} = 1, ~ i=1,2$. Now,
\medskip
\[\|\chi_I T(f_1^{0},f_2^{\infty})\|_{L^{p}(G)}
\le
C\|f_1\|_{\mathcal{M}_{p_1,u_1}(G)}
\|f_2\|_{\mathcal{M}_{p_2,u_2}(G)}\|\chi_{I}\|_{{L^p}(G)}\frac{u_1(I_{-1})|I_{-1}|}{\|\chi_{I_{-1}}\|_{L^{p_1}(G)}}
\left\{ \sum_{j=-\infty}^{-1}
\frac{1}{|I_{j-1}|}
\frac{u_2(I_{j-1})}{\|\chi_{I_{j-1}}\|_{L^{p_2}(G)}}\right\}\]
Multiplying both sides by $\frac{1}{u(I)}$ and taking supremum over \(I\in \mathcal{I}\), using \ref{Equation 4.16}, \ref{Equation 4.17} and Lemma \ref{Lemma 3.1}, we get
\begin{align*}
E_2 &=
\sup_{I\in\mathcal{I}} \frac{1}{u(I)}
\|\chi_I T(f_1^{0},f_2^{\infty})\|_{L^p}(G)\\
&\le
C\|f_1\|_{\mathcal{M}_{p_1,u_1}(G)}
\|f_2\|_{\mathcal{M}_{p_2,u_2}(G)}\sup_{I\in\mathcal{I}} \frac{1}{u(I)}\|\chi_{I}\|_{{L^p}(G)}\frac{u_1(I_{-1})|I_{-1}|}{\|\chi_{I_{-1}}\|_{L^{p_1}(G)}}
\left\{ \sum_{j=-\infty}^{-1}
\frac{1}{|I_{j-1}|}
\frac{u_2(I_{j-1})}{\|\chi_{I_{j-1}}\|_{L^{p_2}(G)}}
\right\}\\
 &\le
C\|f_1\|_{\mathcal{M}_{p_1,u_1}(G)}
\|f_2\|_{\mathcal{M}_{p_2,u_2}(G)}\sup_{I\in\mathcal{I}} \frac{1}{u(I)}\|\chi_{I}\|_{{L^{p_2}}(G)}u_1(I_{-1})|I_{-1}|
\left\{ \sum_{j=-\infty}^{-1}
\frac{1}{|I_{j-1}|}
\frac{u_2(I_{j-1})}{\|\chi_{I_{j-1}}\|_{L^{p_2}(G)}}
\right\}\\ 
&\le
C\|f_1\|_{\mathcal{M}_{p_1,u_1}(G)}
\|f_2\|_{\mathcal{M}_{p_2,u_2}(G)}\sup_{I\in\mathcal{I}} \frac{u_1(I_{-1})}{u(I)}
\left\{ \sum_{j=-\infty}^{-1}\frac{\|\chi_{I}\|_{{L^{p_2}}(G)}}{\|\chi_{I_{j-1}}\|_{L^{p_2}(G)}}u_2(I_{j-1})\frac{|I_{-1}|}{|I_{j-1}|}
\right\}\\
&\le
C\|f_1\|_{\mathcal{M}_{p_1,u_1}(G)}
\|f_2\|_{\mathcal{M}_{p_2,u_2}(G)}\sup_{I\in\mathcal{I}} \frac{u_1(I)u_{2}(I)}{u(I)}\sum_{j=-\infty}^{-1}\frac{|I_{-1}|}{|I_{j-1}|}\\
&\le
C\|f_1\|_{\mathcal{M}_{p_1,u_1}(G)}
\|f_2\|_{\mathcal{M}_{p_2,u_2}(G)}
\end{align*}
Similarly, \(E_3\) admits the same estimate. To estimate \(E_4\), we first observe that
\begin{align*}
|T(f_1^{\infty},f_2^{\infty})(y)|
&\le
C\int_G\int_G
\frac{|f_1^{\infty}(z_1)||f_2^{\infty}(z_2)|}
{(d(y,z_1)+d(y,z_2))^{2}}
\,dz_1dz_2\\
&\le C \int_{G\setminus I_{-1}}\int_{G\setminus I_{-1}}
\frac{|f_1(z_1)||f_2(z_2)|}{(d(y,z_1)+d(y,z_2))^2}
\,dz_1\,dz_2\\
& \le C \sum_{j_{1}=-\infty}^{-1}
\sum_{j_{2}=-\infty}^{-1}
\left\{\int_{I_{j_{1}-1}\setminus I_{j_{1}}}\int_{I_{j_{2}-1}\setminus I_{j_{2}}}
\frac{|f_1(z_1)||f_2(z_2)|}{(d(y,z_1)+d(y,z_2))^2}
\,dz_1\,dz_2\right\}\\
&\le C \prod_{i=1}^{2}
\left(
\sum_{j_{i}=-\infty}^{-1}
\int_{I_{j_{i}-1}\setminus I_{j_{i}}}
\frac{|f_i(z_i)|}{d(y,z_i)}
\,dz_i\right)\\
&\le C \prod_{i=1}^{2}
\left(
\sum_{j_{i}=-\infty}^{-1}
\frac{1}{|I_{j_{i}-1}|}
\int_{I_{j_{i}-1}}
|f_i(z_i)|\,dz_i
\right)\\
&\le C \prod_{i=1}^{2}
\left(
\sum_{j_i=-\infty}^{-1}
\frac{1}{|I_{j_{i}-1}|}
\|\chi_{I_{j_{i}-1}} f_i\|_{L^{p_i}(G)}
\|\chi_{I_{j_{i}-1}}\|_{L^{p_i'}(G)}
\right)\\
&\le C \prod_{i=1}^{2}
\left(
\sum_{j_i=-\infty}^{-1}
\frac{1}{|I_{j_{i}-1}|}
\|\chi_{I_{j_{i}-1}} f_i\|_{L^{p_i}(G)}
\frac{u_i(I_{j_{i}-1})}{u_i(I_{j_{i}-1})}\frac{\|\chi_{I_{j_{i}-1}}\|_{L^{p_i}(G)}}{\|\chi_{I_{j_{i}-1}}\|_{L^{p_i}(G)}}
\|\chi_{I_{j_{i}-1}}\|_{L^{p_i'}(G)}
\right)\\
&\le C \|f_1\|_{\mathcal{M}_{p_1,u_1}(G)}
\|f_2\|_{\mathcal{M}_{p_2,u_2}(G)}
\prod_{i=1}^{2}
\left(
\sum_{j_{i}=-\infty}^{-1}
\frac{u_i(I_{j_{i}-1})}{\|\chi_{I_{j_{i}-1}}\|_{L^{p_i}(G)}}
\right).
\end{align*}
The above inequality again follows from the fact that  $\frac{1}{|I_{ji-1}|}\|\chi_{I_{ji-1}}\|_{L^{p_i'}(G)}\|\chi_{I_{ji-1}}\|_{L^{p_i}(G)} = 1, ~ i=1,2$. Now, 
\[\|\chi_I T(f_1^{\infty},f_2^{\infty})\|_{L^{p}(G)}
\le
C \|f_1\|_{\mathcal{M}_{p_1,u_1}(G)}
\|f_2\|_{\mathcal{M}_{p_2,u_2}(G)}\|\chi_I\|_{{L^p}(G)}
\prod_{i=1}^{2}
\left(
\sum_{j_i=-\infty}^{-1}
\frac{u_i(I_{j_{i}-1})}{\|\chi_{I_{j_{i}-1}}\|_{L^{p_i}(G)}}
\right).\]
Multiplying both sides by $\frac{1}{u(I)}$ and Taking supremum over \(I\in\mathcal{I}\), using Lemma \ref{Lemma 3.1} and condition \ref{Equation 4.17}, we get
\begin{align*}
E_{4}&=\sup_{I\in\mathcal{I}}
\frac{1}{u(I)}
\|\chi_I T(f_1^{\infty},f_2^{\infty})\|_{L^{p}(G)}\\
&\le
C \|f_1\|_{\mathcal{M}_{p_1,u_1}(G)}
\|f_2\|_{\mathcal{M}_{p_2,u_2}(G)}
\sup_{I\in\mathcal{I}}\frac{1}{u(I)}\|\chi_I\|_{{L^p}(G)}
\prod_{i=1}^{2}
\left(
\sum_{j_i=-\infty}^{-1}
\frac{u_i(I_{j_i-1})}{\|\chi_{I_{j_{i}-1}}\|_{L^{p_i}(G)}}
\right)\\
&\le
C \|f_1\|_{\mathcal{M}_{p_1,u_1}(G)}
\|f_2\|_{\mathcal{M}_{p_2,u_2}(G)}
\sup_{I\in\mathcal{I}}\frac{1}{u(I)}
\prod_{i=1}^{2}
\left(
\sum_{j_i=-\infty}^{-1}
\frac{\|\chi_I\|_{{L^pi}(G)}u_i(I_{j_i-1})}{\|\chi_{I_{j_{i}-1}}\|_{L^{p_i}(G)}}
\right)\\
&\le C \|f_1\|_{M_{p_1,u_1}(G)}
\|f_2\|_{M_{p_2,u_2}(G)}
\sup_{I\in\mathcal{I}}
\frac{u_1(I)u_2(I)}{u(I)}\\
&\le C \|f_1\|_{M_{p_1,u_1}(G)}
\|f_2\|_{M_{p_2,u_2}(G)}\\
\end{align*}
Combining the estimates for \(E_1,E_2,E_3\) and \(E_4\), we obtain the desired result.
\end{proof}

\section{Declaration} 
The first author expresses gratitude for the financial assistance provided by the DST INSPIRE fellowship (file no. DST/INSPIRE/03/2023/001183, Code No. IF220246). The second and third authors do not have funding.\\
Consent to Participate declaration: Not applicable.\\
Consent to Publish declaration: Not applicable.\\
Ethics declaration: Not applicable.


\begin{thebibliography}{99}    
\bibitem{38} Behera, B., Jahan, Q.: Wavelet analysis on local fields of positive characteristic. Springer, Singapore (2021).

\bibitem{22}Christ, M., Journé, J.L.: Polynomial growth estimates for multilinear singular integral operators. Acta Math. \textbf{159}, 51–80 (1987). 

\bibitem{21} Coifman, R.R., Meyer, Y.: On commutators of singular integrals and bilinear singular integrals. Trans. Amer. Math. Soc. \textbf{212}, 315–331 (1975).

\bibitem{2}Edwards, R.E., Gaudry, G.I.: Littlewood-Paley and multiplier theory. Springer-Verlag Berlin, Heidelberg, (2012).

\bibitem{11} Fan, D., Yang, D.: Herz-type Hardy spaces on Vilenkin groups and their applications. Sci. China Math. \textbf{43}(5), 481–494 (2000).

\bibitem{9}Fefferman, C., Stein, E.M.:  $H^{p}$ spaces of several variables. Acta Math. \textbf{129}, 137–193 (1972).



\bibitem{24} Grafakos, L., Torres, R.H.: Multilinear Calderón–Zygmund theory. Adv. Math. \textbf{165}(1), 124–164 (2002).

\bibitem{27} Guanghui, L. Shuangping, T.: Estimate for bilinear Calderón–Zygmund operators and their commutators on products of variable exponent spaces. Bull. Korean Math. Soc. \textbf{59}(6), 1471–1493 (2022).

\bibitem{34} He, S., Tao, S.: Bilinear 
$\theta$-type Calderón–Zygmund operators and their commutators on generalized weighted Morrey spaces over RD-spaces. Results Appl. Math. \textbf{26}, 100587 (2025).
\bibitem{Ho} Ho, K.P.: Calderón-Zygmund operators on Morrey and Hardy-Morrey spaces in locally compact Vilenkin groups. P-Adic Numbers Ultrametric Anal. Appl. \textbf{13}(3), 204-214 (2021).
\bibitem{23} Kenig, C.E., Stein, E.M.: Multilinear estimates and fractional integration. Math. Res. Lett. \textbf{6}(1), 1–15 (1999).

\bibitem{13}Kitada, T.: Multipliers for weighted Hardy spaces on locally compact Vilenkin groups. Bull. Aust. Math. Soc. \textbf{48}(3), 441–449 (1993).

\bibitem{7}Kitada, T.: Multipliers on weighted Hardy spaces over certain totally disconnected groups. Int. J. Math. Math. Sci. \textbf{11}(4), 665–674 (1988).

\bibitem{17} Kitada, T., Onneweer, C.W.: Weighted Triebel–Lizorkin spaces on locally compact Vilenkin groups. Math. Nachr. \textbf{168}(1), 191–208 (1994).

\bibitem{18}Kitada, T., Yang, D.: Potential operators in weighted Herz-type spaces over locally compact Vilenkin groups. Acta Math. Hung. \textbf{90}(1), 29–63 (2001).



\bibitem{25} Li, W., Xue, Q., Yabuta, K.: Multilinear Calderón–Zygmund operators on weighted Hardy spaces. Stud. Math. \textbf{199}, 1–16 (2010).

\bibitem{26} Lu, G.: Bilinear 
$\theta$-type Calderón–Zygmund operator and its commutator on non-homogeneous weighted Morrey spaces. Rev. Real Acad. Cienc. Exactas Fis. Nat. - A: Mat. \textbf{115}(1), 16 (2021).

\bibitem{35} Lu, G., Tao, S., Wang, M.: Bilinear $\theta$-type Calderón–Zygmund operators and their commutators on product generalized fractional mixed Morrey spaces. Math. Nachr. \textbf{297}(6), 1988–2005 (2024).

\bibitem{37} Lu, G., Wang, M.: Estimates for bilinear $\omega$-type Calderón–Zygmund operators and their commutator on the product of grand mixed (generalized) Morrey spaces. Anal. Math. \textbf{51}(3), 941–966 (2025).

\bibitem{14} Lu, S.Z., Yang, D.C.: The decomposition of Herz spaces on local fields and its applications. J. Math. Anal. Appl. \textbf{196}(1), 296–313 (1995).

\bibitem{MAL} Maldonado, D. and Naibo, V.: Weighted norm inequalities for paraproducts and bilinear pseudodifferential operators with mild regularity. J. Fourier Anal. Appl. \textbf{15}(2), 218-261 (2009).
\bibitem{3}Morrey, C.B.: On the solutions of quasi-linear elliptic partial differential equations. Trans. Amer. Math. Soc. \textbf{43}(1), 126-166 (1938).

\bibitem{6}Mo, H.X., Han, Z., Yang, L., Wang, X.J.: p-adic singular integrals and their commutators in generalized Morrey spaces. J. Inequal. Appl. \textbf{2019}(1), 65 (2019).

\bibitem{5}Mo, H., Wang, X., Ma, R.: Commutator of Riesz potential in 
p-adic generalized Morrey spaces. Front. Math. China \textbf{13}(3), 633–645 (2018).

\bibitem{4}Nakai, E.: Hardy–Littlewood maximal operator, singular integral operators and the Riesz potentials on generalized Morrey spaces. Math. Nachr. \textbf{166}(1), 95–103 (1994).

\bibitem{16} Onneweer, C.W., Quek, T.S.: Multipliers on weighted $L^p$-spaces
 over locally compact Vilenkin groups. Proc. Amer. Math. Soc. \textbf{105}(3), 622–631 (1989).

\bibitem{10}Pérez, C., Torres, R.H.: Sharp maximal function estimates for multilinear singular integrals. Contemp. Math. \textbf{320}, 323–332 (2003).

\bibitem{19} Quek, T.S., Yang, D.: Calderón–Zygmund operators on weighted weak Hardy spaces in locally compact Vilenkin groups. Math. Nachr. \textbf{225}(1), 123–143 (2001).

\bibitem{1} Taibleson, M.H.: Fourier analysis on local fields. Princeton University Press, Princeton (1975).
\bibitem{Wan} Wang, C.L.: Variants of the Hölder inequality and its inverses. Canad. Math. Bull., \textbf{20}(3), 377-384 (1977).
\bibitem{29} Wang, L., Shu, L.: Multilinear commutators of singular integral operators in variable exponent Herz-type spaces. Bull. Malays. Math. Sci. Soc. \textbf{42}(4), 1413–1432 (2019).

\bibitem{30} Wang, W., Xu, J.: Multilinear Calderón–Zygmund operators and their commutators with BMO functions in variable exponent Morrey spaces. Front. Math. China \textbf{12}(5), 1235–1246 (2017).

\bibitem{36} Wang, S., Xu, J.: Weighted norm inequality for bilinear Calderón–Zygmund operators on Herz–Morrey spaces with variable exponents. J. Inequal. Appl. \textbf{2019}(1), 251 (2019).

\bibitem{15} Yang, D.: Some remarks on singular integrals and power weights on Vilenkin groups. Acta Math. Hung. \textbf{81}(1), 69–88 (1998).
\end{thebibliography}
\end{document}